\documentclass{amsart}

\usepackage[colorlinks=true,linkcolor=blue,citecolor=blue,urlcolor=blue]{hyperref}
\usepackage{bookmark}
\usepackage{amsthm,thmtools,amssymb,amsmath,amscd}

\usepackage{fancyhdr}
\usepackage{esint}

\usepackage{enumerate}

\usepackage{pictexwd,dcpic}

\usepackage{graphicx}

\swapnumbers
\declaretheorem[name=Theorem,numberwithin=section]{thm}
\declaretheorem[name=Remark,style=remark,sibling=thm]{rem}
\declaretheorem[name=Lemma,sibling=thm]{lemma}
\declaretheorem[name=Proposition,sibling=thm]{prop}
\declaretheorem[name=Definition,style=definition,sibling=thm]{defn}
\declaretheorem[name=Corollary,sibling=thm]{cor}
\declaretheorem[name=Assumption,style=remark,sibling=thm]{ass}
\declaretheorem[name=Example,style=remark,sibling=thm]{example}

\numberwithin{equation}{section}

\usepackage{cleveref}
\crefname{lemma}{Lemma}{Lemmata}
\crefname{prop}{Proposition}{Propositions}
\crefname{thm}{Theorem}{Theorems}
\crefname{cor}{Corollary}{Corollaries}
\crefname{defn}{Definition}{Definitions}
\crefname{example}{Example}{Examples}
\crefname{rem}{Remark}{Remarks}
\crefname{ass}{Assumption}{Assumptions}
\crefname{not}{Notation}{Notation}

\renewcommand{\-}{\bar}
\newcommand{\bs}{\backslash}
\newcommand{\cn}{\colon}
\newcommand{\sub}{\subset}

\newcommand{\R}{\mathbb{R}}

\renewcommand{\S}{\mathbb{S}}

\newcommand{\8}{\infty}

\renewcommand{\a}{\alpha}
\renewcommand{\b}{\beta}
\newcommand{\g}{\gamma}
\renewcommand{\d}{\delta}
\newcommand{\e}{\epsilon}
\renewcommand{\k}{\kappa}
\renewcommand{\l}{\lambda}

\newcommand{\s}{\sigma}

\newcommand{\vt}{\vartheta}
\renewcommand{\O}{\Omega}

\newcommand{\G}{\Gamma}

\newcommand{\del}{\partial}

\newcommand{\inpr}[2]{\left\langle #1,#2 \right\rangle}
\newcommand{\fr}[2]{\frac{#1}{#2}}

\DeclareMathOperator{\id}{id}
\DeclareMathOperator{\pr}{pr}

\newcommand{\Theo}[3]{\begin{#1}\label{#2} #3 \end{#1}}
\newcommand{\pf}[1]{\begin{proof} #1 \end{proof}}
\newcommand{\eq}[1]{\begin{equation}\begin{alignedat}{2} #1 \end{alignedat}\end{equation}}

\newcommand{\br}[1]{\left(#1\right)}

\newcommand{\ra}{\rightarrow}

\newcommand{\mt}{\mapsto}

\DeclareMathOperator{\reflectionvector}{V}
\DeclareMathOperator{\reflectionangle}{\delta}

\newcommand{\reflectionmap}[1][\reflectionvector]{\ensuremath{R_{#1}}}
\newcommand{\reflectionset}[2][\reflectionvector]{\ensuremath{{#2}_{#1}}}

\newcommand{\ip}[2]{\ensuremath{\langle{#1},{#2}\rangle}}

\newcommand{\mc}{\mathcal}
\renewcommand{\it}{\textit}
\newcommand{\mrm}{\mathrm}

\newcommand{\hp}{\hphantom}

\protected\def\ignorethis#1\endignorethis{}
\let\endignorethis\relax

\begin{document}

\title[Ancient solutions to curvature flows in the sphere]
 {On the classification of ancient solutions to curvature flows on the sphere}

\author[P. Bryan]{Paul Bryan}
\address{Mathematics Institute, University of Warwick
Coventry, CV4 7AL, England}
\email{p.bryan@warwick.ac.uk}
\author[M.N. Ivaki]{Mohammad N. Ivaki}
\address{Institut f\"{u}r Diskrete Mathematik und Geometrie, Technische Universit\"{a}t Wien,
Wiedner Hauptstr. 8--10, 1040 Wien, Austria}
\email{mohammad.ivaki@tuwien.ac.at}
\author[J. Scheuer]{Julian Scheuer}
\address{Albert-Ludwigs-Universit\"{a}t,
Mathematisches Institut, Eckerstr. 1, 79104
Freiburg, Germany}
\email{julian.scheuer@math.uni-freiburg.de}
\date{\today}

\dedicatory{}

\begin{abstract}
We consider the evolution of hypersurfaces on the unit sphere $\mathbb{S}^{n+1}$ by smooth functions of the Weingarten map. We introduce the notion of `quasi-ancient' solutions for flows that do not admit non-trivial, convex, ancient solutions. Such solutions are somewhat analogous to ancient solutions for flows such as the mean curvature flow, or 1-homogeneous flows. The techniques presented here allow us to prove that any convex, quasi-ancient solution of a curvature flow which satisfies a backwards in time uniform bound on mean curvature must be stationary or a family of shrinking geodesic spheres. The main tools are geometric, employing the maximum principle, a rigidity result in the sphere and an Aleksandrov reflection argument. We emphasize that no homogeneity or convexity/concavity restrictions are placed on the speed, though we do also offer a short classification proof for several such restricted cases.
\end{abstract}

\maketitle

\section{Introduction}
\label{sec:intro}

We consider the evolution of a closed, convex hypersurfaces $M^n$ by
\eq{\label{eq:CurvFlow}
\partial_tx=-F\nu,~ x\cn(-T,0)\times M^n\to \S^{n+1},
}
where \(\S^{n+1}\) is equipped with the round metric of constant curvature $1,$ $\nu$ is the outward unit normal to $M_t=x(t,M^n)$ and $F$ is a function of the principal curvatures of $M_t.$ Whenever we refer to \eqref{eq:CurvFlow}, $F$ will be understood to have the following properties without further mentioning.
\begin{ass} \label{F}
Let $\G_{+}=\{(\k_i)\in\R^n\cn \k_i>0,~ i=1,\ldots,n\}.$ Assume that $F\in C^{\8}(\G_{+}) \cap C^{0}(\-\G_{+})$ is
\begin{enumerate}[(i)]
\item{symmetric},
\item{positive on $\G_{+}$ with $F(0)=0$}  and
\item{increasing, i.e., for all $1\leq i\leq n$ we have
$\fr{\del F}{\del \k_i}\geq 0.$}
\end{enumerate}
\end{ass}
We emphasize that no homogeneity, convexity or concavity restrictions are placed on the speed - all we require is that the speed is parabolic, positive and symmetric (equivalently, the speed is invariant under isometries). It \emph{shrinks} strictly convex hypersurfaces. The assumption $F(0)=0$ ensures that equators are static solutions and that the choice of an \emph{outward} unit normal makes sense for all nontrivial solutions.

We consider \emph{quasi-ancient}, convex solutions of \eqref{eq:CurvFlow}. Here quasi-ancient refers to solutions existing on the same time interval $(-T,0)$ as the maximal flow of strictly convex geodesic spheres, compare \cref{Quasi}. For many common flows, such as the mean curvature flow \(F(\mathcal{W}) = \text{Trace}(\mathcal{W}) = H\), the maximal interval of a flow of geodesic spheres is the infinite time interval \((-\infty, 0)\). However, for flows \(F = H^p\) with \(p \in (0,1)\), the maximal time interval is in fact bounded; see \Cref{quasi}.

The following theorem is the main result of the paper.
\begin{thm}
\label{thm:main}
Let \(M_t\) be a smooth, quasi-ancient, convex solution of \eqref{eq:CurvFlow} where $F$ satisfies \cref{F} and such that
\[\limsup_{t\rightarrow -T}\max_{M_t}H<\infty.\] Suppose that either
\begin{enumerate}[(i)]
\item\label{main1} $M_t$ is strictly convex for all $t \in (-T, 0)$, or
\item \label{main2}$F\in C^1(\bar\Gamma_{+}).$
\end{enumerate}
\noindent
Then $M_t$ is a flow of geodesic spheres.
\end{thm}

In certain situations, such as 1-homogeneous speeds that are either convex, or concave and comparable to \(H\), the theorem is quite easy to prove, and we offer a short proof in these cases in \Cref{sec:concave_convex_homogeneous}. For instance, this should be compared with the result in \cite[Theorem 6.1]{HuiskenSinestrari:10/2015} where a short proof in the case of the mean curvature flow in the sphere may be found. The main idea is to find the correct quantity to which the maximum principle may be applied. In the paper referred to it is the pinching quantity \((|A|^2 - n H^2)/H^2\) used so often for the mean curvature flow. Here we use \(\mathcal{W}/F\) but also remark that the curvature ball estimates from \cite{AndrewsHanLiWei:/2015} may also be used.

The main thrust of the theorem is that we have very strong rigidity in the sphere, and we can prove the result in great generality - \emph{all we require is weak parabolicity of the flow}. At this level of generality, it is quite difficult to obtain suitable uniform regularity estimates, in particular, Evans-Krylov higher regularity estimates are not known for arbitrary non-linear equations, and hence less direct PDE methods are required. We obtain the theorem by first using the rigidity results of \cite{MakowskiScheuer:/2016} to argue that quasi-ancient solutions with bounded \(H\) must limit to an equator backwards in time. Then the Aleksandrov reflection technique developed in \cite{BryanIvaki:08/2015,BryanLouie:04/2016} applies to show the symmetry is preserved under the flow, completing the classification.

Let us remark that one may seek other methods of classification, using more PDE theoretic techniques, but they tend to suffer the drawback that they only apply (without modification) to speeds \(F\) that are non-singular and non-degenerate on the equator. The methods we describe here apply even in the case of \emph{singular or degenerate speeds}. Very simple examples of such speeds are \(H^p\), which are singular for \(0 < p < 1\) and degenerate for \(p > 1\) whenever \(H = 0\), in particular along the equator. We find that the geometric methods are quite appealing and of interest in their own right, even in cases where PDE methods are applicable, these methods offering a powerful, complementary alternative to direct PDE techniques, and supplanting them in cases where such methods are unknown.

One limitation of our method is that at this stage, we assume a bound on the mean curvature \(H\). Although there are certainly solutions emanating from convex polyhedra (and hence with unbounded \(H\)), it is not at all clear to us whether such solutions can be quasi-ancient. In \Cref{sec: examples}, we give two examples (for non-geometric flows) showing what could go wrong. The first converges to a lune (intersection of hemispheres), with unbounded \(H\) and the second converges to an equator, but is not a family of shrinking spheres. Since these examples are not geometric flows, they are not of the type considered here, but they do show that parabolic flows on the sphere may admit ancient solutions that are not a family of shrinking geodesics spheres.

It is also worth comparing the techniques and results here to the Euclidean case. The latter does not exhibit such strong rigidity, and other ancient, convex solutions may occur. Indeed, for the curve shortening flow in the plane, the classification \cite{DaskalopoulosHamiltonSesum:/2010} states that solutions are either shrinking circles or the Angenent oval (also known as the paper clip solution) \cite{Angenent:/1992}. In higher dimensions a characterization of when ancient, convex mean curvature flows are shrinking spheres have been given \cite{HaslhoferHershkovits:/2016,HuiskenSinestrari:10/2015}, and one generally expects a greater variety of ancient, convex solutions to exist.

Convex, ancient solutions are of interest as they arise as rescalings of singularities, cf.~ \cite{HuiskenSinestrari:01/1999}, of mean convex, mean curvature flow, cf.~\cite{HuiskenSinestrari:09/1999, White:10/2002}. Singularity models for mean curvature flow in the sphere have been studied in \cite{Nguyen:07/2015}, where mean convex singularities with an ambient curvature dependent pinching condition are classified according to a blow-up yielding ancient solutions in Euclidean space.

This paper is laid out as follows: In \Cref{prelim}, we establish some notation and basic equations to be used in this paper. In \Cref{quasi}, we introduce the notion of quasi-ancient solutions and establish some facts about such solutions. In \Cref{sec:concave_convex_homogeneous}, we give a short classification result for 1-homogeneous concave or convex speeds and for the curve shortening flow. In \Cref{sec:rigidity}, we use the rigidity results from \cite{MakowskiScheuer:/2016} to show that the backwards limits of convex quasi-ancient solutions with bounded \(H\) are equators. In \Cref{sec:reflection}, we use an Aleksandrov reflection technique to complete the classification. In \Cref{sec: examples}, we look at two examples of non-geometric flows for which there exist convex, ancient solutions that are not shrinking spheres, illustrating the necessary restriction to geometric flows and justifying our use of geometric methods.

\section*{Acknowledgment}
The authors would like to thank Knut Smoczyk and the Institut f\"{u}r Differentialgeometrie at Leibniz Universit\"at for hosting a research visit where part of this work took place. The first author was a Riemann Fellow at the Riemann Center for Geometry and Physics, Leibniz Universit\"{a}t and was also supported by the EPSRC on a Programme Grant entitled ``Singularities of Geometric Partial Differential Equations'' reference number EP/K00865X/1. The work of the second author was supported by Austrian Science Fund (FWF) Project M1716-N25 and the European Research Council (ERC) Project 306445. All three authors would like to thank Slack.com and Overleaf.com for providing wonderful platforms for collaboration. The figure in this paper was created using the wonderful SageMath Cloud (cloud.sagemath.com).

\section{Preliminaries}\label{prelim}
\label{sec:prelim}
It is well known, e.g. \cite[Chapter 2]{Gerhardt:/2006}, that a curvature function as in \cref{F} also can be considered to depend on the Weingarten map,
\[F=F(\mc{W})=F(h^i_j),\]
in case of which the derivatives of the speed \(F\) are written as
\[
F^{i}_{j} = \frac{\partial F}{\partial h^{j}_{i}}.
\]
It is also possible to consider $F$ as a function of the metric and second fundamental form,
\[
F(g, h) = F(g^{ik} h_{kj}),
\]
where we write
\[
F^{ij} = \frac{\partial F}{\partial h_{ij}}, \quad F^{ij,kl} = \fr{\partial^2F}{\partial h_{kl} \partial h_{ij}}.
\]
These derivatives are related by
\[F^{ij}=g^{ik}F^j_k.\]
Let us also define the operator
\[
\Box = F^{ij} \nabla^2_{ij},
\]
where $\nabla$ is the Levi-Civita connection of the induced metric $g$. Occasionally $\nabla$ will denote the Levi-Civita connection of the round metric of $\S^n$, in which case we will explicitly say so. 
\begin{lemma} \label{lem: basi ev}
The following evolution equations hold.
\eq{
 \label{eq:delt_weingarten_box}
\partial_t h_i^j &= \Box h_i^j + F^{kl} (h^2)_{kl} h_i^j - (F^{kl}h_{kl} - F) (h^2)_i^j + F^{kl,rs}\nabla_i h_{kl}\nabla^j h_{rs} \\
& \quad +  (F + F^{kl}h_{kl}) \delta_i^j - F^{kl}g_{kl} h_i^j,
}
\eq{\label{eq:delt_speed} \partial_t F = \Box F + FF^{ij}(h^2)_{ij} +  FF^{ij}g_{ij}.}
\end{lemma}
\pf{The proofs are standard; see \cite[Lemma~2.4.3, Lemma~2.3.4]{Gerhardt:/2006}.}

We will occasionally have to work with graphs \(u\) over an equator $E \subset \mathbb{S}^{n+1}$ around a point $e\in\S^{n+1}.$ In geodesic polar coordinates around $e$ the spherical metric takes the form
\eq{\label{SphMetric}d\-s^2=d\rho^2+\vt(\rho)^2\s_{ij}(x)dx^idx^j,}
where $\rho$ is the radial distance to $e,$ $\vt(\rho)=\sin(\rho)$ and $\s$ the round metric on the equator $E\simeq \S^n.$ Hence, if a hypersurface $M$ is given as a graph of a function $u$ over $E$,
\[M=\{(\rho,(x^i))\in \S^{n+1}\cn \rho=u(x),~ x\in E\},\]
a straightforward computation yields the following representation of the Weingarten map in terms of the function $u,$ namely
\eq{\label{graph h}h^i_j=\fr{\vt'}{v\vt}\d^i_j+\fr{\vt'}{v^3\vt^3}\nabla^iu\nabla_ju-\fr{g^{ik}}{v}\nabla^2_{kj}u,}
where $v^{-1} = \partial_\rho \cdot \nu.$ Covariant derivatives as well as index raising are performed with respect to $\s_{ij};$ see, for example, \cite[(3.82)]{Scheuer:05/2015}.

A one-parameter family of graphs satisfying the flow \eqref{eq:CurvFlow} is a solution of
\begin{equation}
\label{eq:GraphCurvFlow}
\fr{\del}{\del t}u=-F\br{\fr{\vt'}{v\vt}\d^{i}_{j}+\fr{\vt'}{v^{3}\vt^{3}}\nabla^{i}u\nabla_{j}u-\fr{g^{ik}}{v}\nabla^{2}_{kj}u}v=\Phi(x,u,\nabla u,\nabla^{2}u);
\end{equation}
see \cite[p.~98-99]{Gerhardt:/2006}.
\section{Ancient and quasi-ancient Solutions}\label{quasi}
We are interested in solutions with maximal possible lifetime. To understand this maximal time, we define $T_S$ to be the lifespan of \it{the strictly convex spherical solution} of \eqref{eq:CurvFlow}. By the strictly convex spherical solution we mean a family of geodesic spheres shrinking under the flow \eqref{eq:CurvFlow} collapsing to a point at time $t=0$ and existing on the maximal interval \((-T_S, 0)\). \(T_S\) can be finite or infinite, as the following lemma shows.
\begin{lemma}The following assertions hold:
\label{lem:spherical_existence_time}
\begin{description}
  \item[i] Let $p\in(0,1)$ and consider \eqref{eq:CurvFlow} with speed \(F = f^p\) where \(f\) is strictly increasing, concave and 1-homogeneous. Then the strictly convex spherical solution exists only on a finite time interval \((-T_S,0)\) with \(0 < T_S < \infty\), collapses to a point at \(t=0\) and converges to an equator at \(t=-T_S\).
  \item[ii] Let $F$ be $1$-homogeneous, then the strictly convex spherical solution of \eqref{eq:CurvFlow} is ancient.
\end{description}
\end{lemma}
\begin{proof}
(i)~We may assume that $f(1,\ldots,1)=n.$
Since $F$ is constant on a geodesic sphere, the evolution equation (\ref{eq:delt_speed}) for a flow of geodesic spheres yields
$$\fr{d}{dt}f\geq \fr 1n f^{p+2}+ nf^{p},$$
where we used \cite[Lemma~2.2.19, Lemma~2.2.20]{Gerhardt:/2006}.
Finite lifespan forward in time follows immediately. Integration over an interval $(a,b) \subset (-T_S, 0)$ yields
$$0\leq f^{1-p}(a)\leq f^{1-p}(b)-n(1-p)(b-a).$$
Allowing $a\ra-\8$ gives the finite time existence backwards in time.

(ii)~Assume $F(1,\ldots,1)=n$ and let $r$ be the radius of the spherical solution. Then \eqref{eq:CurvFlow} yields
\[\dot{r}=-n\cot{r}.\]
The maximal spherical solution satisfies
\[r(0)=0\quad\text{and}\quad r(-T)=\fr{\pi}{2}\]
and hence integration of this ODE yields $T=\8.$
\end{proof}
\begin{lemma}
Let $x$ be a non-equatorial, convex solution of \eqref{eq:CurvFlow}, defined on the open interval $(-T,0),$ where $0$ is the collapsing time, then $T\leq T_S.$
\end{lemma}
\begin{proof}
Suppose $T>T_S+\e$ for some $\e>0$. $M=M_{-T_S-\fr{\e}{2}}$ bounds a convex body $\hat{M}$ and is strictly contained in an open hemisphere due to the classical paper \cite{CarmoWarner:/1970}. Then there exists a geodesic sphere $S$ with $\hat{M}\sub\hat{S}.$ By the avoidance principle, the flow with initial hypersurface $M$ collapses before the spherical flow thus contradicting $T>T_S$.
\end{proof}
In view of this lemma, the following definition is reasonable.
\begin{defn}\label{Quasi}
A convex solution of \eqref{eq:CurvFlow} defined on an interval $(-T,0)$ is called \it{quasi-ancient}, if $T=T_S$. For convenience we also call the static equatorial solution quasi-ancient.
\end{defn}
 By the definition, convex ancient solutions are also quasi-ancient.
For convex speeds of homogeneity $1$, the following proposition gives a bound on mean curvature backwards in time for ancient solutions, using a Harnack inequality. For genuine quasi-ancient solutions ($T_S < \infty$), the Harnack inequality (whenever it holds) does not give such a bound since we cannot send \(s \to -\infty\) as in the proof. For example, a Harnack inequality holds for the flows $F = H^p$, $0<p<1$; see \cite[Theorem 1]{BIS1} and from \cref{lem:spherical_existence_time}, $T_S < \infty$. One could envisage backwards limits as convex polyhedra and hence with unbounded \(H\), but it is not clear that these solutions exist on the maximal time interval \((-T_S, 0)\), i.e., we do not know if they can be quasi-ancient solutions. For the classification of quasi-ancient solutions, we thus make the additional assumption that \(H\) is bounded, and defer the question of whether solutions with unbounded \(H\) can exist as an interesting study for a later date.

For a strictly convex hypersurface, define $(b^{ij})$ to be the inverse of $(h_{ij}).$
\begin{prop}
\label{cor:boundedH}
Suppose $F$ is a strictly monotone, \(1\)-homogeneous and convex curvature function. Then any strictly convex ancient solution of \eqref{eq:CurvFlow} satisfies
\[\partial_t F-b^{ij}\nabla_i F \nabla_j F \geq 0.\]
Therefore, for any $t_0 > 0$ and all $t\le -t_0$ we have
$H(\cdot,t)\leq c(t_0),$
where $c(t_0)<\infty$ depends only on $M_{-t_0}.$
\end{prop}
\begin{proof}
For any $t>s$, the  Harnack estimate of \cite[Theorem 1]{BIS1} implies that
$$\partial_t F-b^{ij}\nabla_i F\nabla_j F+\frac{1}{2}\frac{F}{t-s}>0.$$
Allowing $s\to-\infty$ proves the first claim.

In particular, for strictly convex, ancient solutions, $b^{ij} > 0$ and hence $F(\cdot, t)$ is a non-decreasing function in $t$. For the second claim then, observe that for any 1-homogeneous, convex $F$ we have \[F\ge \frac{F(1,\ldots,1)}{n}H,\]
see \cite[Lemma~2.2.20]{Gerhardt:/2006}. Therefore, strictly convex, ancient solutions satisfy
\[H(\cdot,t)\leq \frac{n}{F(1,\ldots,1)}F(\cdot,t)\leq \frac{n}{F(1,\ldots,1)}F(\cdot,-t_0). \]
for all \(t \leq -t_0\).
\end{proof}
\section{Convex, Concave and Homogeneous Speeds}\label{sec:concave_convex_homogeneous}
In this section, we give a quick classification of ancient solutions in three special cases.
\begin{thm}\label{thm: quick}
Let $F$ be strictly monotone, $1$-homogeneous, and either convex, or concave with $H \leq c F$ for some $c>0$. Then any strictly convex ancient solution of the flow \eqref{eq:CurvFlow} with speed $F$ is a family of contracting geodesic spheres.
\end{thm}
\pf{
We may assume \(F\) is normalized so that \(F(1, \ldots, 1) = n\). We will give the details for \(F\) convex. The concave case is similar, and we remark below where the proof deviates slightly.

Convexity and 1-homogeneity implies that $F\geq \tfrac{F(1,\ldots,1)}{n} H;$ see \cite[Lemma~2.2.20]{Gerhardt:/2006}. The chosen normalization then implies,
\begin{equation}
\label{eq:speed_ineq}
\frac{\kappa_{\min}}{F} \leq \frac{\kappa_{\min}}{H} \leq \frac{1}{n}.
\end{equation}

Now we prove the reverse inequality, \(\tfrac{\kappa_{\min}}{F} \geq \tfrac{1}{n}\). The tensor
$w^j_i=\fr{h^j_i}{F}$
satisfies the evolution equation
\begin{align*}
\del_tw^j_i&=\Box w^j_i-2 F^{kl}g_{kl}w^j_i+2\d^j_i+\fr{1}{F}F^{kl,rs}\nabla_ih_{kl}\nabla^jh_{rs}\\
    &\hp{=}-2F^{kl}\nabla_kh^j_i\nabla_l\br{\fr{1}{F}}-2\fr{h^i_j}{F^3}F^{kl}\nabla_k F\nabla_l F.
\end{align*}

At a critical point of $w^i_i$ we have
$$0=\nabla_k \br{\fr{h^i_i}{F}}=\fr{\nabla_k h^i_i}{F} + h^i_i\nabla_k\br{\fr{1}{F}},$$
and hence the second line above vanishes:
\begin{multline*}
-2F^{kl}\nabla_kh^i_i\nabla_l\br{\fr{1}{F}} - 2\fr{h^i_i}{F^3}F^{kl}\nabla_k F\nabla_l F \\
= 2F^{kl} \left(Fh^i_i\nabla_k \br{\fr{1}{F}} \nabla_l \br{\fr{1}{F}}  - \fr{h^i_i}{F^3}\nabla_k F\nabla_l F\right) = 0.
\end{multline*}

Homogeneity, convexity and the normalization \(F(1, \ldots, 1) = n\) imply that
\[
F^{kl}g_{kl} \leq n, \quad \text{ and }\quad  F^{kl,rs} \geq 0.
\]
Fixing an index \(i\) and working in normal coordinates centered on a local minimum we have
\[
\frac{1}{F} F^{kl,rs} \nabla_i h_{kl} \nabla^i h_{rs} = \frac{1}{F} F^{kl,rs} \nabla_i h_{kl} \nabla_i h_{rs}
\]
is non-negative. Thus at a spatial minimum,
\begin{align*}
\partial_t \br{w^i_i-\fr{1}{n}} \geq \Box\br{w^i_i-\fr{1}{n}} - 2n\br{w^i_i-\fr{1}{n}}.
\end{align*}
By the maximum principle, for any $s<0$, we obtain the bound
\begin{align*}
w^i_i-\fr 1n \geq c_s e^{-2n(t-s)}
\end{align*}
for $t\geq s$, and where
\[
c_s := \min \frac{\kappa_{\min}}{F}(\cdot, s)- \frac{1}{n}.
\]
Moreover, we have the uniform bounds
\[
-\frac{1}{n} \leq c_s \leq 0
\]
from equation \eqref{eq:speed_ineq}.

Taking the limit $s\to -\infty$ we then obtain for all $t < 0$,
\begin{equation}
\label{eq:mp_speed_ineq}
\frac{\kappa_{\min}}{F} \geq \frac{1}{n}.
\end{equation}

Combining equations \eqref{eq:speed_ineq} and \eqref{eq:mp_speed_ineq} we obtain
\[
\frac{\kappa_{\min}}{F} \leq \frac{\kappa_{\min}}{H} \leq \frac{1}{n} \leq \frac{\kappa_{\min}}{F}.
\]
Hence, $H\equiv n \kappa_{\min}$ and the flow is by totally umbilical, closed hypersurfaces, and hence by geodesic spheres.

In the concave case, we obtain
\begin{align*}
w^i_i-\fr 1n \leq d_s e^{-2n(t-s)},
\end{align*}
where
$
d_s := \max \frac{\kappa_{\max}}{F}(\cdot, s) - \frac{1}{n}\geq 0.
$
Using the additional assumption \(\tfrac{H}{F} \leq c\), we obtain that
\[d_s\leq \max\frac{H(\cdot, s)}{F}-\frac{1}{n} \leq c-\frac{1}{n}.\]
Now by letting $s\to-\infty,$ we conclude that $H \equiv n \kappa_{\max}$ and the claim follows.
}
\begin{rem}
The bounds,
\[
\begin{split}
\frac{\kappa_{\max}}{F} &\leq \frac{1}{n}, \quad \text{$F$ concave} \\
\frac{\kappa_{\min}}{F} &\geq \frac{1}{n}, \quad \text{$F$ convex}
\end{split}
\]
may also be obtained from the curvature ball estimates of \cite[Theorem 1.1]{AndrewsHanLiWei:/2015}, which imply that any strictly convex solution of the flow with a 1-homogeneous, concave speed $F$ satisfies
\[
\frac{\kappa_{\max}}{F}(x,t) \leq \frac{1}{n}+d_se^{-2n(t-s)},
\]
with \(d_s\) as in the proof above. For convex speeds, the curvature ball estimates imply that
\[
\frac{\kappa_{\min}}{F}(x,t)\geq \frac{1}{n}+c_se^{-2n(t-s)},
\]
where \(c_s\) is from the proof above. The proof applying the maximum principle to \(h^i_i/F\) above, though weaker than the curvature ball estimates, is significantly simple and suffices for our purposes.
\end{rem}
\begin{rem}
In virtue of the evolution equations in Lemma \ref{lem: basi ev}, for strictly increasing, 1-homogeneous, convex speeds $F$ with $F\in C^1(\bar\Gamma_{+})$,  any non-equatorial, convex, ancient solution of (\ref{eq:CurvFlow}) is in fact strictly convex. To see this, apply the strong maximum principle to the evolution equation of $F$ to conclude that $F$ must be strictly positive for any time. On the other hand, the last term in the evolution equation of $\kappa_{\min}$ is bounded below by
\[F+F-n\kappa_{\min}\geq F+H-n\kappa_{\min}\ge F>0.\]
Therefore, the statement of Theorem \ref{thm: quick} holds under the weaker assumption of convexity of ancient solutions rather than strict convexity.
\end{rem}
We end this section by providing a short proof for the classification result of \cite{BryanLouie:04/2016}.
\begin{thm}\cite{BryanLouie:04/2016}
The only convex, ancient solutions of the curve shortening flow on the sphere are shrinking geodesic circles or equators.
\end{thm}
\begin{proof}
Suppose $\gamma_t$ is a convex, ancient solution to the curve shortening flow. Let $s$ denote the arc-length parameter of $\gamma_t$. By Lemma \ref{lem: basi ev} and $\frac{d}{dt}ds=-\kappa^2ds$ we have
\begin{align*}
\frac{d}{dt}\left(\left(\int_{\gamma_t}\kappa ds\right)^2+\left(\int_{\gamma_t}ds\right)^2\right)=2\left(\left(\int_{\gamma_t}\kappa ds\right)^2-\int_{\gamma_t}ds\int_{\gamma_t}\kappa^2ds\right)\leq 0.
\end{align*}
Therefore,
\[q(t) = \left(\int_{\gamma_t}\kappa ds\right)^2+\left(\int_{\gamma_t}ds\right)^2 \]
is a non-increasing function. By \cite[Lemma 4.1]{BryanLouie:04/2016}, $\lim\limits_{t\to-\infty}\int_{\gamma_t}\kappa ds=0;$ therefore,
\[
\lim_{t\to -\infty} q(t) = 4 \pi^2.
\]
On the other hand, by the Gauss-Bonnet theorem and the isoperimetric inequality \cite{Rado:10/1935}, $q(t) \geq 4\pi^2$ and hence $$q(t)\equiv 4\pi^2.$$
By the characterization of the equality cases, $\gamma_t$ are either shrinking geodesic circles or equators.
\end{proof}
\section{Rigidity and Backwards Limit}\label{sec:rigidity}
Now we turn to more general speeds \(F\). The aim of this section is to prove that for a quasi-ancient solution of \eqref{eq:CurvFlow} \it{with bounded mean curvature} for $t\ra -T_S$ the backwards limit of the flow hypersurfaces $M_t$  is an equator. We will use the rigidity result of \cite{MakowskiScheuer:/2016} to achieve this. For convenience, we state this result:
\begin{thm}\label{Rigidity}\textsc{\cite[Theorem 1.1]{MakowskiScheuer:/2016}}
Let $ n\geq 1$ and $\hat{M}\subset \S^{n+1}$ be a weakly convex body in a hemisphere. Let $x_0\in \S^{n+1}$ be such that $\hat{M}$ is contained in the closed hemisphere $\mc{H}(x_0)$ with equator $\mc{S}(x_{0})$. Suppose that $\hat{M}$ satisfies an interior sphere condition at all points $p\in \hat{M} \cap \mc{S}(x_0)$. Then either $\hat{M}$ is equal to $\mc{H}(x_0)$ or $\hat{M}$ is contained in an open hemisphere.
\end{thm}
Here $\hat{M}$ is a weakly convex body in a hemisphere $\mc{H}(x_0),$ if it is a compact set with non-empty interior and for every two points $p,q\in\hat{M}$ there exists a minimizing geodesic connecting $p$ with $q,$ while being contained in $\hat{M}.$ $\hat{M}$ satisfies an interior sphere condition at $p\in \del\hat{M}$ with radius $R,$ if there exists a geodesic ball $B_R$ with radius $R,$ such that
\[\del B_R\cap\del \hat{M}=\{p\}\quad~\text{and}~B_R\sub\mrm{int}(\hat{M}). \]
Note that the points $p$ in \cref{Rigidity} are automatically boundary points of $\hat{M}.$

In what follows, in addition to the notion of ``weakly convex body", we will also need the notions of ``weakly convex set", ``convex set" and ``convex body" for which we refer the reader to \cite[Definition 3.2]{MakowskiScheuer:/2016}.

The next lemma is the first step in providing the assumptions of \cref{Rigidity}.
\begin{lemma}\label{ISC}
Let $x$ be a convex, quasi-ancient solution of \eqref{eq:CurvFlow} with backwards bounded mean curvature. Then there holds:
\begin{enumerate}
  \item For all $t_0<0$ there exists a uniform radius $R>0,$ such that the enclosed convex bodies $\hat{M}_t,$ $-T_S<t\leq t_0,$ of the flow hypersurfaces $M_t$ satisfy a uniform interior sphere condition with radius $R.$
  \item For every $y_0\in\mrm{int}~\hat{M}_{t_0}$ the hypersurfaces $M_t,$ $-T_S<t\leq t_0,$ can be written as a graph in geodesic polar coordinates around $y_0$ and the corresponding graph functions satisfy uniform $C^2$-estimates.
\end{enumerate}
\end{lemma}
\pf{
Fix an interior point $y_0\in \mrm{int}~\hat{M}_{t_0}.$ Since for a contracting flow the enclosed convex bodies of the flow hypersurfaces are decreasing, they are increasing backwards in time. By the proof of \cite[Lemma~3.9]{MakowskiScheuer:/2016} there exists a closed hemisphere $\mc{H}(x_0),$ such that
$$\hat{M}_t\sub\mc{H}(x_0).$$
In our situation all hypersurfaces $M_t,$ $-T_S<t\leq t_0,$ satisfy
$$B_{\e}(y_0)\sub \mrm{int}~\hat{M}_t$$
and
$$B_{\e}(\hat{y}_0)\sub \hat{M}_t^c$$
with a uniform $\e,$ where $\hat{y}_0$ denotes the antipodal point of $y_0.$
Now we prove the two claims.

\begin{enumerate}
  \item Consider the stereographic projection with $\hat{y}_0$ corresponding to infinity. The image hypersurfaces are then strictly convex hypersurfaces in Euclidean space with uniformly bounded second fundamental form. Blaschke's rolling theorem (see \cite{Blaschke:/1956}) gives the interior sphere condition.
  \item Write the $M_t$ as graphs in geodesic polar coordinates around $y_0,$
$$M_t=\{(\rho,x^i)\cn \rho=u(t,x^i)\}.$$
Due to \eqref{SphMetric}, on the set in which $M_t$ range, the metrics $g_{ij},$ $\-g_{ij}=\sin^2\rho\s_{ij}$ and $\s_{ij}$ are all equivalent.
In view of \cite[Theorem 2.7.10]{Gerhardt:/2006}, for all convex hypersurfaces $M_t$ the quantity
\begin{align*}
v^2=1+\-g^{ij}\nabla_iu\nabla_ju
\end{align*}
is uniformly bounded by a constant which only depends on $\e.$
Hence by the equivalence of norms, $M_t$ are uniformly $C^1$-bounded in the sense that the corresponding functions $u(t,\cdot)$ are uniformly $C^1(\S^n)$-bounded. Recalling equation \eqref{graph h}, due to the curvature estimates we obtain uniform $C^2(\S^n)$-estimates for \(u\).
\end{enumerate}
}
\begin{cor}\label{Backlimit}
Let $x$ be a convex and quasi-ancient solution of \eqref{eq:CurvFlow} with backwards bounded mean curvature. Then there exists a unique backwards limiting hypersurface $M_{-T_S}$ and the flow hypersurfaces $M_t$ converge to $M_{-T_S}$ in $C^{1,\b},$ $0<\b<1,$ in the sense that for a common graph representation as in Lemma \ref{ISC} there holds
$$u(t,\cdot)\ra u(-T_S,\cdot)$$
in the norm of $C^{1,\b}(\S^n).$
\end{cor}
\pf{
In view of the point-wise monotonicity of $u(t,\cdot)$ backwards in time, we obtain a point-wise limit. The $C^{1,\b}$-convergence follows from compactness.
}
\begin{thm}
\label{thm:backwardslimit}
The hypersurface $M_{-T_S}$ defined in Corollary \ref{Backlimit} is an equator.
\end{thm}
\pf{
Since the convex bodies $\hat{M_t}$ are increasing backwards in time and due to the uniform convergence of $M_t$ to $M_{-T_S},$ the set
$$\hat{M}_{-T_S}:=\overline{\bigcup_{t<0}\hat{M}_t}$$
is a compact body with
$$\del \hat{M}_{-T_S}=M_{-T_S}.$$
Since $\mrm{int}(\hat{M}_{-T_S})$ is a convex set, it is especially a weakly convex set in a hemisphere. Thus $\hat{M}_{-T_S}$ is a weakly convex body in a hemisphere. The proof of \cite[Lemma 6.1]{MakowskiScheuer:/2016} can literally be applied to show that $\hat{M}_{-T_S}$ satisfies a uniform interior sphere condition as well.
We can apply \cite[Theorem 1.1]{MakowskiScheuer:/2016} and obtain that $\hat{M}_{-T_S}$ is either strictly contained in an open hemisphere or is equal to a closed hemisphere. The first alternative is not possible since the solution is quasi-ancient.{\footnote{On the contrary, suppose $\hat{M}_{-T_S}$ is contained in an open hemisphere, then there exists an open geodesic ball $B$ of radius $\rho\leq \pi/2-\e$, for some $\e>0,$ enclosing $\hat{M}_{-T_S}$. Its boundary sphere $\del B$ has a maximal finite time existence $T<T_S-c_{\e}$ for some constant $c_{\e}$. For all $t>-T_S$ we have $M_t\subset B$. Hence the flow starting from $M_t$ exists only for $-t\leq T.$  Letting $t\to -T_S$ yields a contradiction.}} We conclude that $\del \hat{M}_{-T_S}=M_{-T_S}$ is an equator of $\S^{n+1}.$
}
\section{Aleksandrov Reflection and Classification}\label{sec:reflection}
In this section, we use the result of Theorem \ref{thm:backwardslimit} to classify convex and quasi-ancient solutions of contracting curvature flows on \(\S^{n+1}\) as either equators or shrinking geodesic spheres. The proof uses Aleksandrov reflection as in \cite{BryanIvaki:08/2015, BryanLouie:04/2016} to show that the symmetry of the backwards limit is preserved along the flow. Here we give a very general version with minimal assumptions on the flow: all we require is that the flow limits to an equator at $-T_S$ in \(C^0\), with uniform \(C^2\) bounds, and that the flow is geometric and parabolic, compare \cref{F}.
\subsection{Notation}
We consider $\S^{n+1}$ to be embedded into $\R^{n+2}$ by the standard inclusion.
Let
$$e=e_{n+2}=\br{0,\ldots,0,1}\in \R^{n+2}.$$
For a vector $V\in \S^{n+1}$ let
$H_{V}^{\pm}=\{x\in\R^{n+2}\cn \pm\inpr{x}{V}>0\},$ and
for a set $S\in \R^{n+2}$ define
$S^{\pm}_{V}=S\cap H^{\pm}_{V}.$ We denote the equator in $\S^{n+1}$ around $e$ by
$E=e^{\perp}\cap \S^{n+1}.$
Furthermore, $\d_{V}$ denotes the signed angle that $V$ makes with the hyperplane $e^{\perp},$
$\d_{V}=\arcsin\inpr{V}{-e}.$
The reflection map across the hyperplane $V^{\perp}$ is denoted by
\begin{align*}
R_{V}\cn\R^{n+2}&\ra \R^{n+2}\\
            x&\mt x-2\inpr{x}{V}V,
            \end{align*}
which is an isometry of $\R^{n+2}.$
Finally, for $x\in \S^{n+1}\bs\{\pm e\},$
let $\g_{x}$ denote the unique minimizing geodesic in $\S^{n+1}$ from $e$ to $-e$ that passes through $x$ and then define
\begin{align*}
\pr(x)=\begin{cases} y, & x\in \S^{n+1}\bs\{\pm e\}\\
                E, & x=e~\text{or}~x=-e,\end{cases}
                \end{align*}
where $y\in E$ is the unique element on the image of $\g_{x}$ lying in $E.$

The reader may find it useful to refer to \Cref{fig:reflection} for the arguments in this section.
\begin{figure}[htb]
\centering
\includegraphics[width=.9\linewidth]{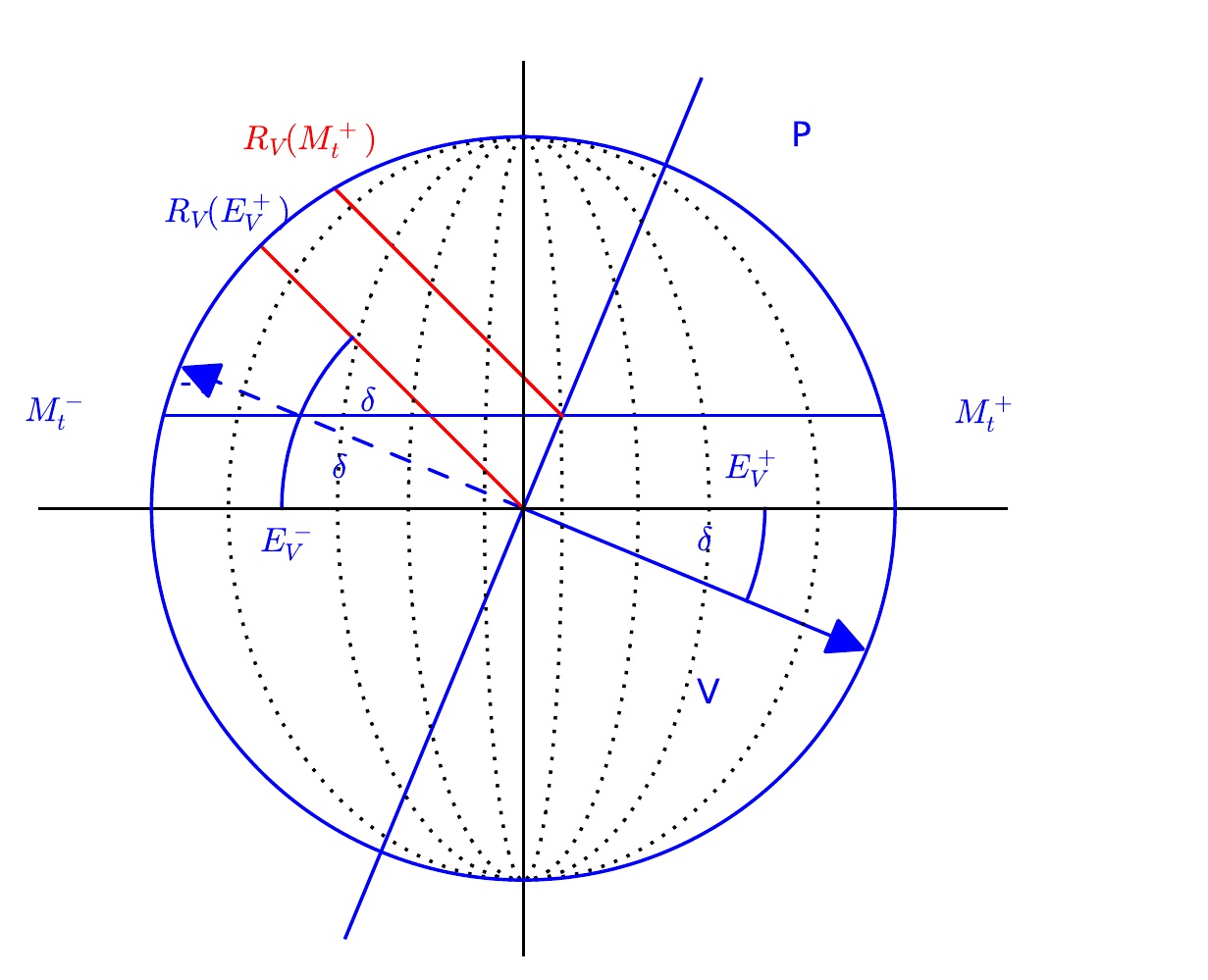}
\caption{Reflection in the $(e, \reflectionvector)$-plane showing the reflected equator, \(M_t\), \(\reflectionmap(M_t)\) and some geodesics through the north pole (dotted lines).}
\label{fig:reflection}
\end{figure}
\subsection{Star-shaped hypersurfaces and one-sided reflections}
\Theo{rem}{starshaped}{
Recall a nonempty set $S\in \S^{n+1}$ to be star-shaped around $e,$ if $\pm e\notin S$ and if every minimizing geodesic from $e$ to $-e$ hits $S$ at most once, i.e.,
$$\forall x\in E\cn \# \br{\mrm{Im}~\g_{x}\cap S}\leq 1.$$
In this case we have a well-defined graph function of $S,$
\begin{align*}f_{S}\cn \pr{S}&\ra \br{-\fr{\pi}{2},\fr{\pi}{2}}\\
            \s&\mt \arcsin\inpr{x_{\s}}{e}, \end{align*}
where $x_{\s}$ is the unique preimage of $\s$ in $S$ under the projection $\pr.$ We obtain a parametrization of the set $S$ via the correspondence
$$x_{\s}=(\s,f_{S}(\s)).$$
}
\Theo{defn}{ReflectsAbove}{
Let $S,~T$ be two star-shaped sets around $e$ with the corresponding functions $f_{S}$ and $f_{T}.$
\begin{enumerate}
  \item We say that \it{S lies above T}, denoted by $S\geq T,$ if
$$\forall \s\in \pr{S}\cap\pr{T}\cn f_{S}(\s)\geq f_{T}(\s).$$
  \item Let $V\in \R^{n+2}$ and suppose $R_{V}(S)$ and $T$ are star-shaped around $e.$ We say that \it{S one-sided reflects above T}, if
$$R_{V}(S^{+}_{V})\geq T^{-}_{V}.$$
\end{enumerate}
}

\Theo{lemma}{Starshaped Reflection}{
Let $V\in\S^{n+1}$ and $0\leq\d_V\leq\d_{0}<\fr{\pi}{4}.$ Then
\begin{description}
  \item[i] $R_{V}(E)$ is star-shaped around $e.$
  \item[ii] There exists a constant $c=c(\d_{0})>0$ with the following property. If a closed and convex hypersurface $M\sub\S^{n+1}$ is star-shaped around $e$ and satisfies
$$|f_{M}|\leq c,$$
then the reflection $R_{V}(M)$ is star-shaped around $e.$
\end{description}
}
\pf{
(i)~Since $R_{V}(E)$ is also an equator, the only way not to be star-shaped around $e$ would be
$e\in R_{V}(E). $
But then, letting $e=R_{V}(x)$ with $x\in E,$ and using the Cauchy-Schwartz inequality, we obtain
$$1=\inpr{R_V(x)}{e }=2\sin\delta_V\inpr{x}{V}\leq 2\sin\delta_V\cos\d_V,$$
which is impossible in the given range of $\d_{V}.$

(ii) Let us put
\[2r=\min\limits_{\{V\cn 0\leq \d_V\leq \d_0\}} \mrm{dist}(e,R_V(E))>0.\]
The minimum is precisely achieved for vectors $V$ with $\d_V=\d_0$ and thus $r$ depends on $\delta_0.$ Suppose $S_{r}(\pm e)$ are the geodesic spheres of radii $r$ with centers at $\pm e.$ Note that $S_{r}(e)$ is contained in the interior of one of the open hemispheres formed by $R_V(E)$, and $S_{r}(-e)$ is contained in the interior of the opposite hemisphere.

Write $E_{\pm\varepsilon}$ for the geodesic spheres at heights $\pm\varepsilon$. By continuity, there exists $c(\d_0)$ such that for all $0<\varepsilon\leq c(\d_0)$, $S_{\frac{r}{2}}(e)$ is contained in the interior of the smallest spherical cap cut by $R_V(E_{\varepsilon})$, and $S_{\frac{r}{2}}(-e)$ is contained in the interior of the smallest spherical cap cut by $R_V(E_{-\varepsilon}).$

Assume that $-c(\d_0)\leq f_M\leq c(\d_0)$. Then using $E_{\pm c(\d_0)}$ as barriers we can see that one of the bodies enclosed by $R_V(M)$ contains \emph{only} $e$ in its interior and the other contains $-e$ in its interior. Now convexity of $R_V(M)$ implies that it is star-shaped around $e.$
}
It is easily seen that the equator $E$ one-sided reflects above itself with respect to $V^{\perp}$, if $0<\d_{V}<\fr{\pi}{4}.$ Now we show that this is also true for any convex hypersurface $M$, which is $C^{1}$-close to $E$.
\Theo{prop}{ReflectedM}{
Suppose $V\in\S^{n+1}$ and $0<\d_{1}\leq\d_{V}\leq \d_{0}<\fr{\pi}{4}$. There exists $\a=\a(\d_{0},\d_{1})>0$ with the following property. If $M$ is a closed and convex hypersurface that is star-shaped around $e$ and
$$0\leq f_{M},\quad |f_{M}|_{C^{1}}\leq\a,$$
then $M$ one-sided reflects above itself.
}
\pf{
Assume without loss of generality that $f_M\leq c(\delta_0),$ where $c(\delta_0)$ is given in the previous lemma.
We will prove
\[R_V(M^{+}_V)\geq M_V^{-}.\]
Suppose there exist $x\in M^{+}_V$ and $y\in M_V^{-}$ with
\[\pr(R_V(x))= \pr(y)\in \pr(R_V(M^{+}_V))\cap\pr(M_V^{-}).\]
Consider the projection map
\begin{align*}
\pr&\cn S_{\varepsilon}\subset \mathbb{R}^{n+2}\ra E\subset \mathbb{R}^{n+1}\\
 				\pr(a)&=\fr{a-\ip{a}{e}e}{|a-\ip{a}{e}e|},
 \end{align*}
where $S_{\varepsilon}$ is the compact set that is made by removing two opposite caps of height $0<\varepsilon<c(\delta_0)$ from the unit sphere at $e,-e$. Since $\pr$ is smooth on $S_{\varepsilon}$, it is Lipschitz and  we have the estimate
\[|\pr(a)-\pr(b)|\leq l_{\varepsilon}|a-b|,\quad\forall a,b\in S_{\varepsilon} \]
 for some constant $0<l_{\varepsilon}<\infty$.
In addition, we have
\begin{align*}
\inpr{x}{e}& = \sin(f_M(\pr(x)),\quad \forall x\in M,\\
2\arcsin\br{\frac{|\pr(a)-\pr(b)|}{2}}&=d(\mrm{pr}(a),\pr(b)),\\
\arcsin(z)&\leq \frac{\pi}{2}z,\quad \forall~ 0\leq z \leq 1.
\end{align*}
Therefore, we obtain
\[
\begin{split}
\inpr{x}{e} - \inpr{y}{e} &\geq -|\sin(f_M(\pr(x))) - \sin(f_M(\pr(y)))|\\
&\geq -  d(\pr(x), \pr(y))|f_M|_{C^1(\mathbb{S}^n)} \\
&= - 2 \arcsin(\frac{1}{2} |\pr(x) - \pr(y)|)|f_M|_{C^1} \\
&\geq - \frac{\pi}{2}|\pr(x) - \pr(y)||f_M|_{C^1}.
\end{split}
\]
Let us put $\alpha=\min\{\frac{2\sin\d_1}{\pi l_{\varepsilon}},c(\delta_0)\}.$ To finish the proof, note that
\begin{align*}\ip{R_V(x)}{e}-\ip{y}{e}&=\ip{R_V(x)}{e}-\ip{x}{e}+\ip{x}{e}-\ip{y}{e}\\
					&\geq 2\ip{x}{V}\sin\d_{V}-\frac{\pi}{2}|\pr(x)-\pr(y)||f_M|_{C^1}
\\                  &=|R_V(x)-x|\sin\d_{V}-\frac{\pi}{2}|\pr(x)-\pr(R_V(x))||f_M|_{C^1}\\
                    &\geq \left(\sin\d_{1}-\frac{\pi}{2}l_{\varepsilon}\alpha\right)|R_V(x)-x|\geq 0
                    \end{align*}
and hence,
\[
\arcsin(\ip{R_V(x)}{e}) \geq \arcsin(\ip{y}{e})
\]
as required in \Cref{ReflectsAbove}.
}
\subsection{Reflecting quasi-ancient curvature flows}
The final aim of this section is to show that for a solution of a curvature flow defined on an interval $(-T_{S},0),$ which converges backwards in time to $E$ in $C^{1}$ and has bounded curvature, the flow hypersurfaces are invariant under reflection about every hyperplane $V^{\perp}$ with $V\in E.$ Hence we somehow need to let $\d_1$ go to zero in the preceding proposition. In fact, until now we have only obtained that for a given $\d_1>0,$ there exists $T_{\d_1}$ such that all flow hypersurfaces in an interval $(-T_{S},T_{\d_1})$ one-sided reflect above themselves. Thus we cannot yet deduce that any flow hypersurface has this property for $\d_1=0.$ To eliminate the dependence of the interval $(-T_{S},T_{\d_1})$ on $\d_1$, we will use the parabolic maximum principle.
\Theo{lemma}{MPReflection}{
Let the closed and convex hypersurfaces $M_{t},$ $-T_{S}<t<0,$ satisfy the parabolic curvature flow equation \eqref{eq:CurvFlow} under either assumption (i) or (ii) in \cref{thm:main} and suppose
\[M_{t}\ra E,\quad t\ra -T_{S},\]
in the sense that $e$ lies in the enclosed convex bodies by $M_{t}$ for all sufficiently small times and the graph functions satisfy
\[0\leq f_{M_t}\quad\text{and}\quad \lim_{t\ra -T_{S}}|f_{M_{t}}|_{C^{1}}=0.\]
 Then there exists $T>-T_{S},$ such that for all $V\in \S^{n+1}$ satisfying
 \[0<\d_{V}\leq \fr{\pi}{8},\]
 $M_{t}$ one-sided reflects above itself for all $t\in (-T_{S},T).$
 }
\pf{
We may assume $M_t$ is not the equator $E$ and that $f_{M_t}\leq \pi/4$ for all $t\leq T_{\ast}.$ By \cref{Starshaped Reflection} and the $C^{1}$-convergence of $M_{t}$ we find a time $T\in (-T_{S}, T_{\ast}],$ such that $M_{t}$ and $R_{V}(M_{t})$ are both star-shaped around $e$ and $e$ also lies in the convex body of $R_V(M_t)$ for all $t\in (-T_{S},T].$ We show that this $T$ has the desired property.

By \cref{ReflectedM}, we know that for a given $V\in \S^{n+1}$ with (suppressing the subscript $V$)
\[0<\d\leq \fr{\pi}{8},\]
 there exists $-T_{S}<\~T_{\d}\leq T_{\ast},$ such that $M_{t}$ one-sided reflects above itself for all $t\in (-T_{S},\~T_{\d}).$ Pick a $-T_{S}<T_{\d}<\~T_{\d}$ and define the domain
\[\O=\bigcup_{t\in(T_{\d},T)}\br{\pr(M_{t}^{-})\cap \pr(R(M_t^{+})}\times \{t\}\sub E\times(-T_{S},T).\]
Recall from \eqref{eq:GraphCurvFlow} that the time dependent radial graph functions of $M_{t}$ satisfy
\begin{align*}\label{FullyNonlinear}
\fr{\del}{\del t}w=-F\br{\fr{\vt'}{v\vt}\d^{i}_{j}+\fr{\vt'}{v^{3}\vt^{3}}\nabla^{i}w\nabla_{j}w-\fr{g^{ik}}{v}\nabla^{2}_{kj}w}v,
\end{align*}
where $\vt=\vt\br{w}=\sin\br{w},$ $w(\cdot,t)=\pi/2-f_{M_t}(\cdot),$ $v^2=1+\vt^{-2}|\nabla w|^2$
and the connection and the norm $|\cdot|$ are those belonging to the spherical metric $\s_{ij}.$ Since the Weingarten operator is invariant under ambient isometries, the reflected hypersurfaces $R_{V}(M_{t})$ with the corresponding graph functions $u$ satisfy the same curvature flow equation.

If the claim of the lemma were false, then there existed a time $t_{\ast}\in (T_{\d},T)$ and a point $x_{\ast}\in \pr(M_{t_{\ast}}^{-})\cap\pr(R(M_{t_{\ast}}^{+})),$ such that
\eq{\label{MPReflection1} u(x_{\ast},t_{\ast})>w(x_{\ast},t_{\ast})\geq \frac{\pi}{4}.}
Formally we would like to apply the parabolic maximum principle to conclude that $u$ must remain below $w,$ since this is the case at the initial time $T_\d$ and also at the boundaries $\del\br{\pr(M_t^{-})\cap\pr(R(M_t^{+}))}\sub\pr\br{V^{\perp}\cap M_t}.\footnote{Note that for two sets $A$ and $B$ we have
 \[\del(A\cap B)\sub \del A \cup \del B.\]
 Since $\pr$ is a diffeomorphism from $M_t$ to $E$ and $M_t$ is star-shaped,
 \[\del(\pr(M_t^{-}))=\pr(\del M_t^{-})= \pr(V^{\perp}\cap M_t)\]
Similarly, $\del(\pr(R(M_t^{+})))=\pr(R(\del M_t^{+}))= \pr(R(V^{\perp}\cap M_t))=\pr(V^{\perp}\cap M_t)$.}$

However, in this situation the application of the standard comparison principles to equations of the form
\[\dot{w}=\Phi(x,w,\nabla w,\nabla^2w)\]
is not straightforward; note to the restriction that $F$ is generally only defined on non-negative definite endomorphisms. Nevertheless, assuming (\ref{MPReflection1}), we obtain a contradiction using the following argument.

A slight rewriting of the second fundamental form of $R(M_t)$ described by the graph function $u$ yields, cf.~\eqref{graph h},
\begin{align*}h^i_j&=\fr{\vt'}{\sqrt{\vt^2+|\nabla u|^2}}\d^i_j+\fr{\vt'}{\br{\vt^2+|\nabla u|^2}^{\fr 32}}\nabla^iu\nabla_ju-\fr{\vt g^{ik}}{\sqrt{\vt^2+|\nabla u|^2}}\nabla^2_{kj}u\\
        &=\fr{\vt g^{ik}}{\sqrt{\vt^2+|\nabla u|^2}}\br{\fr{\vt'}{\vt}g_{kj}+\fr{\vt'}{\vt\br{\vt^2+|\nabla u|^2}}g_{km}\s^{mr}\nabla_ru\nabla_ju-\nabla^2_{kj}u}\\
        &\equiv \fr{\vt(u)g^{ik}}{\sqrt{\vt^2(u)+|\nabla u|^2}}A_{kj}(x,u,u,\nabla u,\nabla^2u),
        \end{align*}
where $\vt=\vt(u)$ and we set
$$A_{kj}(x,z,u,p,r)=\fr{\vt'(z)}{\vt(z)}g_{kj}(p,u)+\fr{\vt'(z)}{\vt(z)\br{\vt^2(z)+|p|^2}}g_{km}(p,u)\s^{mr}p_rp_j-r_{kj}.$$
Here the dependence on $x$ is hidden in $\s^{mr}.$ Moreover, $g_{kj}(p,u)=p_k p_j+\vt^{2}(u)\s_{kj}$ and we artificially introduced $z$ in the argument in order to distinguish increasing behavior from decreasing behavior in $u.$ Since $\cot(z)$ is decreasing in $z$, we have
\begin{align}\label{iv}
A_{ij}\br{x,z_2,u,\nabla u,\nabla^2u}X^iX^j<A_{ij}\br{x,z_1,u,\nabla u,\nabla^2u}X^iX^j
\end{align}
for all $z_1,z_2\in(0,u(x,t))$ with $z_2>z_1$ and all non-zero vectors $X\in TM_t$.

Also write
\begin{align*}
  h^i_j(x,z,u,p,r)&=\fr{\vt(z)g^{ik}(p,z)}{\sqrt{\vt^2(z)+|p|^2}}A_{kj}(x,z,u,p,r).
\end{align*}
Note that if $u$ graphs a convex hypersurface, (\ref{iv}) implies that
\[h^i_j(x,z,u(x,t),\nabla u(x,t),\nabla^2u(x,t))>0,\quad \forall~ 0<z< u(x,t)\]
as a self-adjoint endomorphism and thus $h^i_j$ lies in the domain of $F.$ Write
\[\Phi(x,z,u,p,r)=-F(h^i_j(x,z,u,p,r))\sqrt{1+\vt^{-2}(z)|p|^2},\quad \forall~ 0<z< u(x,t).\]
Choose
 \eq{\label{mon phi}\l>\sup_{(x,t)\in E\times[T_{\delta},T]}\sup_{z\in Q(x,t)}\fr{\del\Phi}{\del z}(x,z,u(x,t),\nabla u(x,t),\nabla^2 u(x,t)),}
 where for $(x,t)\in E\times[T_{\delta},T]$ we defined
 $Q(x,t)=\{z:  \frac{\pi}{4}\leq z\leq u(x,t)\}.$
Such a choice of $\lambda$ is possible since when $z$ is away from $0$, $\Phi$ is smooth.

In view of \eqref{MPReflection1}, there exists $(x_0,t_0)\in\-\O$ not lying on the parabolic boundary of $\O$ such that for
\[\~u=ue^{-\l t}\quad\text{and}\quad \~w=we^{-\l t},\]
there holds
 \begin{align*}
 0<\~u(x_0,t_0)-\~w(x_0,t_0)=\max_{\-\O}(\~u-\~w).
 \end{align*}
Therefore,
\begin{align}\label{is}
[w(x_0,t_0),u(x_0,t_0)]\subset Q(x_0,t_0).
\end{align}
Note that at $(x_0,t_0)$ we have $\dot{\~u}\geq \dot{\~w},$ $\nabla u=\nabla w$ and $\nabla^2u\leq \nabla^2w.$ Hence using (\ref{is}), (\ref{mon phi}) and our choice of $\lambda$, we obtain, at $(x_0,t_0)$, that
\begin{align*}\br{\dot{u}-\Phi(x_0,u,u,\nabla u,\nabla^2u)}e^{-\l t_0}&=\dot{\~u}-\Phi(x_0,u,u,\nabla u,\nabla^2 u)e^{-\l t_0}+\l\~u\\
        &> \dot{\~u}-\Phi(x_0,w,u,\nabla u,\nabla^2 u)e^{-\l t_0}+\l\~w \\
        &\geq \dot{\~w}-\Phi(x_0,w,w,\nabla w,\nabla^2w)e^{-\l t_0}+\l\~w\\
        &=\br{\dot{w}-\Phi(x_0,w,w,\nabla w,\nabla^2 w}e^{-\l t_0},
    \end{align*}
 a contradiction, since both sides are zero in view of the flow equation.
}
\begin{proof}[Proof of \cref{thm:main}]
From \cref{MPReflection} we have \(\reflectionmap(\reflectionset{(M_t)}^+) \geq \reflectionset{(M_t)}^-\) everywhere for all \(t \in (-T_S, -T)\) and any \(\reflectionangle \in (0,\reflectionangle_0)\). By continuity therefore, sending \(\reflectionangle \to 0\) we have \(\reflectionmap(\reflectionset{(M_t)}^+) \geq \reflectionset{(M_t)}^-\) for all \(t \in (-T_S, -T)\)  and any \(\reflectionvector\) satisfying \(\ip{\reflectionvector}{e} = 0\).

Now we need some properties of $\reflectionmap[\reflectionvector]$ following from the fact that $\ip{\reflectionvector}{e} = 0$:
\begin{itemize}
\item $\reflectionmap^2 = \id$,
\item $S \geq T \Rightarrow \reflectionmap(S) \geq  \reflectionmap(T)$,
\item $\reflectionmap = \reflectionmap[-\reflectionvector]$, and
\item $\reflectionset{S}^{\pm} = \reflectionset[-\reflectionvector]{S}^{\mp}$.
\end{itemize}
Thus we obtain
\begin{align*}
\reflectionset{(M_t)}^+ &= \reflectionmap(\reflectionmap(\reflectionset{(M_t)}^+)) \geq \reflectionmap(\reflectionset{(M_t)}^-) \\
&= R_{-V}((M_t)^+_{-V}) \geq \reflectionset[-\reflectionvector]{(M_t)}^-\\
&= \reflectionset{(M_t)}^+.
\end{align*}
So we must have equality all the way through and hence the middle line implies
\[
R_{-V}((M_t)^+_{-V}) = \reflectionset[-\reflectionvector]{(M_t)}^-
\]
for any $\reflectionvector$.

Therefore, \(M_t\) is invariant under \(\reflectionmap\) for any \(\reflectionvector\) satisfying \(\ip{\reflectionvector}{e} = 0\), hence it is a geodesic sphere for every \(t \in (-T_S, -T)\) and thus it is a geodesic sphere for every \(t \in (-T_S, 0)\).
\end{proof}
\section{Non-geometric counterexamples}\label{sec: examples}
The examples in this section highlight the necessity of geometric invariance, without which our methods can not be applied. First, we have a flow that is parabolic - but not geometric - admitting a convex, ancient solution for which \(H\) becomes unbounded. As mentioned earlier, we do not yet know whether such singular behavior can occur for isotropic, geometric flows. Certainly this is not possible whenever a differential Harnack inequality holds and the spherical solutions are ancient as in \cref{cor:boundedH}.
\begin{example}
Let $\bar{\gamma}_t$, evolve by the curve shortening flow in $\mathbb{R}^2$, and let $\gamma_t \subset \mathbb{S}^2$ be the inverse gnomonic projection (see, e.g., \cite{BesauWerner:11/2014}) of \(\bar{\gamma_t}\). The spherical radial functions
\[\rho:\mathbb{S}^1\times[0,T)\to \mathbb{R}\]
evolve by
\begin{equation}\label{eq: angenent oval}
\partial_t\rho(\cdot,t)=-\kappa \frac{\sqrt{\sin^2\rho+\rho_{\theta}^2}}{\sin\rho}\frac{\sin^2\rho+\rho_{\theta}^2}{\tan^2\rho+(\tan\rho)_{\theta}^2}(\cdot,t)
\end{equation}
where, $\kappa(\cdot,t)$ is the curvature of the curve $\gamma_t$ with radial function $\rho(\cdot,t).$

To obtain the evolution equation, in polar coordinates we can express the curvature as follows:
\[\kappa=\frac{-\rho_{\theta\theta}\sin\rho+2\rho_\theta^2\cos\rho+\cos\rho\sin^2\rho}{(\sin^2\rho+\rho_{\theta}^2)^{\frac{3}{2}}}.\]
Write $\bar{\rho}(\cdot,t)$ for the radial function of $\bar{\gamma}_t$. We recall from \cite[p.~8]{BesauWerner:11/2014} that
$\bar{\rho}=\tan\rho.$ Using this formula and the expression of $\kappa$ we can write the curvature of $\bar{\gamma}_t$,  $\bar{\kappa}(\cdot,t),$ as follows:
\[\kappa=\left(\frac{\bar{\rho}^2+\bar{\rho}_{\theta}^2}{(1+\bar{\rho}^2)(\sin^2\rho+\rho_{\theta}^2)}\right)^{\frac{3}{2}}\bar{\kappa}=\left(\frac{\bar{\rho}^2+1}{\bar{h}^2+1}\right)^{\frac{3}{2}}\bar{\kappa}.\]
Here $\bar{h}=\frac{\bar{\rho}^2}{\sqrt{\bar{\rho}^2+\bar{\rho}_{\theta}^2}}$ is the support function of $\bar{\gamma}.$
Therefore,
\begin{align*}
\partial_t\bar{\rho}&=-\bar{\kappa}(1+\bar{\rho}^2)\left(\frac{\bar{\rho}^2+\bar{\rho}_{\theta}^2}{(1+\bar{\rho}^2)(\sin^2\rho+\rho_{\theta}^2)}\right)^{\frac{3}{2}}\frac{\sqrt{\sin^2\rho+\rho_{\theta}^2}}{\sin\rho}\frac{\sin^2\rho+\rho_{\theta}^2}{\tan^2\rho+(\tan\rho)_{\theta}^2}\\
&=-\bar{\kappa}(1+\bar{\rho}^2)^{\frac{3}{2}}\left(\frac{\bar{\rho}^2+\bar{\rho}_{\theta}^2}{(1+\bar{\rho}^2)(\sin^2\rho+\rho_{\theta}^2)}\right)^{\frac{3}{2}}\frac{\sqrt{\sin^2\rho+\rho_{\theta}^2}}{\bar{\rho}}\frac{\sin^2\rho+\rho_{\theta}^2}{\tan^2\rho+(\tan\rho)_{\theta}^2}\\
&=-\bar{\kappa}\frac{\sqrt{\bar{\rho}^2+\bar{\rho}_{\theta}^2}}{\bar{\rho}}.
\end{align*}
The curve shortening flow in $\mathbb{R}^2$ has non-trivial ancient solutions, the Angenent ovals. Thus, there exists a non-spherical, convex, ancient solution to the flow (\ref{eq: angenent oval}). This ancient solution converges backwards in time to a lune (the intersection of two hemispheres) with a pair of antipodal points at which \(\kappa\) is unbounded.
\end{example}
The second example exhibits a (non-geometric) flow for which the symmetry of the backwards limit is not preserved. That is, the backwards limit is an equator, yet the flow is not by geodesic spheres. This example is not as troubling as the first, since the non-geometric nature of the flow prohibits the use of the Aleksandrov reflection technique in \Cref{sec:reflection}, but it does illustrate the necessity of geometric invariance for our methods.
\begin{example}
Using the notation of the previous example, let $\bar{\gamma}_t$ evolve by the affine normal flow in $\mathbb{R}^2:$
\begin{align*}
\partial_t\bar{\rho}&=-\bar{\kappa}^{\frac{1}{3}}\frac{\sqrt{\bar{\rho}^2+\bar{\rho}_{\theta}^2}}{\bar{\rho}}.
\end{align*}
Let \(\gamma_t\) be the inverse gnomonic projection. Then the spherical radial functions evolve by
\[\rho:\mathbb{S}^1\times[0,T)\to \mathbb{R}\]
\begin{equation}\label{eq: ellipse}
\partial_t\tan\rho(\cdot,t)=-\kappa^{\frac{1}{3}} \frac{\sqrt{\sin^2\rho+\rho_{\theta}^2}}{\sin\rho}(\cdot,t).
\end{equation}
Origin-centered ellipses are ancient solutions to the affine normal flow in $\mathbb{R}^2$ (see, for example, \cite{Ivaki:01/2016}). Thus, there exists a non-spherical, convex, ancient solution to the flow (\ref{eq: ellipse}). This solution converges backwards in time to an equator.
\end{example}
\bibliographystyle{amsplain}
\bibliography{Bibliography}
\end{document}